\documentclass[11pt]{amsart}

\usepackage{amsfonts}
\usepackage{amsthm}
\usepackage{amssymb}
\usepackage{amsmath}
\usepackage{url, hyperref}
\usepackage{graphicx}
\usepackage[english]{babel}
\usepackage[margin=1.25in]{geometry}

\newtheorem{theo}{\bf Theorem}[section]
\newtheorem{lemma}{\bf Lemma}[section]

\newcommand{\spn}{{\rm span}}

\newcommand{\Z}{{\mathbb Z}}
\newcommand{\N}{{\Bbb N}}
\newcommand{\Q}{{\Bbb Q}}
\newcommand{\R}{{\mathbb R}}

\newcommand{\conv}{\rm conv\; }

\newcommand{\bea}{\begin{eqnarray*}}
\newcommand{\eea}{\end{eqnarray*}}
\newcommand{\be}{\begin{eqnarray}}
\newcommand{\ee}{\end{eqnarray}}

\newcommand{\vol}{\mbox{\rm vol}\,}
\newcommand{\lin}{\mbox{lin}\,}
\newcommand{\ve}{\boldsymbol}

\newcommand{\interior}{\mbox{\rm int}\,}
\newcommand{\frob}{\mathrm{F}}
\newcommand{\dfrob}{\mathrm{g}}

\numberwithin{equation}{section}

\begin{document}

\title[Integer Points in Knapsack Polytopes]{Integer Points in Knapsack Polytopes\\ and\\ $s$-covering Radius}
 \author{Iskander Aliev}
 \address{School of Mathematics and Wales Institute of Mathematical and Computational Sciences, Cardiff University, Senghennydd Road, CARDIFF, Wales, UK}
 \email{alievi@cf.ac.uk}

 \author{Martin Henk}
\address{Fakult\"at f\"ur Mathematik, Otto-von-Guericke
 Universit\"at Mag\-deburg, Universit\"atsplatz 2, D-39106-Magdeburg,
 Germany}
\email{martin.henk@ovgu.de, eva.linke@ovgu.de}

\author{Eva Linke}


\begin{abstract}
 Given a matrix  $A\in \Z^{m\times n}$ satisfying certain
 regularity assumptions, we consider for a positive integer $s$ the set ${\mathcal F}_s(A)\subset \Z^m$ of all
 vectors  ${\ve b}\in \Z^m$ such that the associated {\em knapsack polytope}
\bea
P(A,{\ve b})=\{{\ve x}\in \R^n_{\ge 0}: A {\ve x}={\ve b}\}\,
\eea
contains at least $s$ integer points. In this paper we investigate the structure of the set ${\mathcal
  F}_s(A)$ using the concept of $s$-covering radius. In particular, in the case $m=1$ we prove an optimal
lower bound for the $s$-Frobenius number, which is the largest integer $b$ such that $P(A,b)$ contains less than $s$ integer points.
\end{abstract}



\maketitle


\section{Introduction and statement of results}
\label{intro}
Let $A\in\Z^{m\times n}$, $1\leq m<n$, be an integral
$m\times n$ matrix satisfying
\begin{equation}
\begin{split}
{\rm i)}&\,\, \gcd\left(\det(A_{I_m}) : A_{I_m}\text{ is an $m\times
    m$ minor of }A\right)=1, \\
{\rm ii)}&\,\, \{{\ve x}\in\R^n_{\ge 0}: A\,{\ve x}={\ve 0}\}=\{{\ve 0}\}
\end{split}
\label{assumption}
\end{equation}
and let ${\ve b}\in \Z^m$. The {\em knapsack
  polytope} $P(A,{\ve b})$ is defined as
\be
P(A,{\ve b})=\{{\ve x}\in \R^n_{\ge 0}: A {\ve x}={\ve b}\}\,.
\label{P}
\ee
Note that on account of \eqref{assumption} ii), $P(A,{\ve b})$ is
indeed a polytope (or empty).

Let $s$ be a positive integer. Consider the set ${\mathcal F}_s(A)$ of integer vectors ${\ve b}$ such
that the knapsack polytope (\ref{P}) contains at least $s$ integer points, i.e.,
\begin{equation*}
{\mathcal F}_s(A)=\{{\ve b}\in\Z^m : \#(P(A,{\ve b})\cap\Z^n)\ge s\}.
\end{equation*}
Integer points in knapsack polytopes are a classical topic of study in combinatorics (see e.g. \cite{Alf1} and \cite{Alf}) and in the integer linear programming (see e.g. \cite{AH_Siam} and \cite{ASch}).
In this paper we will study the combinatorial structure of 
${\mathcal F}_s(A)$ using  results from Minkowski's geometry of
numbers and discrete geometry (see e. g. \cite{peterbible} and
\cite{GrLek}).

First we address the
special case $m=1$.
The matrix $A$ now has the form $A={\ve a}^\intercal$, where ${\ve a}=(a_1,
a_2, \ldots, a_n)^\intercal$ is an integer vector  and  \eqref{assumption} i) says that
 $\gcd(a_1, a_2,
\ldots, a_n)=1$, i.e., ${\ve a}$ is a {\em primitive} integer vector.
Due to the second assumption \eqref{assumption} ii)
we may assume that all entries of ${\ve a}$ are positive.  The set ${\mathcal F}_s({\ve a}^\intercal)$ in this case contains all consecutive integers larger than the
$s$-{\em Frobenius number}
\bea
\frob_s({\ve a}):=\max\{b\in \Z: \#(P({\ve a}^\intercal, b)\cap \Z^n)<s \}\,
\eea
associated with vector ${\ve a}$.
The $s$-Frobenius numbers have been introduced and studied by
Beck and Robins in \cite{BeckRobins}. For more recent results please
see \cite{AFH},
\cite{BeckKifer}, \cite{FS} and \cite{ShallitStankewicz}.
Thus when $m=1$ we have
\begin{equation}
\interior(\frob_s({\ve a})+\R_{\ge 0})\cap\Z =\frob_s({\ve a})+\Z_{>0}\subset {\mathcal F}_s({\ve
  a}^\intercal),
\label{eq:frob_cone}
\end{equation}
where $\interior(\cdot)$ denotes the interior of the set.

Note that $\frob_1({\ve a})$ is the classical {\em Frobenius number} associated with the integers $a_1,
a_2, \ldots, a_n$. The general problem of
finding $\frob_1({\ve a})$,  the
{\em Frobenius problem}, is well known to be NP-hard (Ram\'{\i}rez
Alfons\'{\i}n \cite{Alf1, Alf}). Kannan \cite{Kannan} proved that for fixed $n$ the Frobenius problem can be solved in polynomial time.
The proof is based on results from the algorithmic geometry of numbers and the integer programming.


In practice, computing Frobenius numbers remains a difficult computational problem even when dimension $n$ is fixed. Thus the upper
and lower bounds for  $\frob_1({\ve a})$ are of special interest.  In
terms of the input vector ${\ve a}$,
all known upper bounds for $\frob_1({\ve a})$ (see, e.g., Erd\"os and Graham \cite{EG} and Fukshansky and Robins \cite{Lenny}) can be represented in the form
\be
\frob_1({\ve a}) \leq {\rm c}_n\, {\ve a}^\intercal {\ve a},
\label{general_upper_bound}
\ee
for a certain constant ${\rm c}_n$ depending on the dimension.
The quadratic order of
(\ref{general_upper_bound}) with respect to ${\ve a}$ is in general
best possible  (see, e.g., \cite{AH_Siam} and
\cite{Arnold06}).

On the other hand, Aliev and Gruber \cite{AG} proved that Frobenius numbers possess an
optimal lower bound. Indeed
\begin{equation}
\frob_1({\ve a})\ge \vartheta_1(S_{n-1}) \left(\prod_{i=1}^n a_i\right)^\frac{1}{n-1}-\sum_{i=1}^n a_i
\,,
\label{bounds_lower}
\end{equation}
where  $\vartheta_1(S_{n-1})$ is the {\em absolute inhomogeneous
  minimum} of an $(n-1)$-dimensional simplex $S_{n-1}$.

Interestingly, the $s$-Frobenius number can be bounded from above by $\frob_1({\ve a})$ plus a term of the same order in ${\ve a}$ as the main term in the lower bound (\ref{bounds_lower}).
The following upper bound on the $s$-Frobenius number was
recently obtained in \cite{AFH}
\begin{equation*}
\begin{split}
\frob_s({\ve a})&\leq
\frob_1({\ve a}) + (s-1)^\frac{1}{n-1}\,\left((n-1)!\prod_{i=1}^n a_i\right)^\frac{1}{n-1}.
\end{split}
\end{equation*}

The first goal of the present paper is to obtain an optimal lower bound for  $\frob_s({\ve a})$ that applies for all $s$.
To this end, we will need a generalization of the quantity $\vartheta_1(S_{n-1})$ for a slightly more general
setting.

By a {\em lattice} $L\subset \R^n$ we will understand a discrete submodule of
the Euclidean space $\R^n$. Given an $d$-dimensional lattice $L$, its {\em determinant} $\det(L)$ is the $d$-dimensional volume (i.e. $d$-dimensional
Lebesgue measure) of the parallelepiped spanned by the vectors of a basis of $L$. By a {\em convex body} $K \subset\spn_\R(L)$ we will understand a compact convex set with nonempty interior
in the Euclidean space $\spn_\R(L)\subset \R^n$ spanned by the lattice $L$.

For a lattice
$L\subset\R^n$ and a convex body $K\subset\spn_\R(L)$ let
\begin{equation}
\begin{split}
 \mu_s(K,L)  =\min\{\mu > 0 : &\text{ for all } {\ve x}\in\spn_\R(L) \text{ there
   exist distinct } {\ve y}_1,\dots,{\ve y}_s\in L \\ &\text{ such that } {\ve x}\in
 {\ve y}_i+\mu K\ \forall\ 1\leq i\leq s\}
\end{split}
\label{eq:def_covering}
\end{equation}
be  the smallest
positive number such that any ${\ve x}\in\spn_\R(L)$ is covered by at least $s$
lattice translates of $\mu_s K$. $\mu_s(K,L)$ is called the
\emph{$s$-covering radius} of $K$ with respect to $L$.
For $s=1$ we get the well-known covering radius for which we refer to
\cite{peterbible} and \cite{GrLek}.

Following Gruber \cite{Peter}, given a convex body $K\subset \R^n$ we
define 
\bea
\vartheta_s(K)=\inf\frac{\mu_s(K,L)}{\det(L)^{1/n}}\,,
\eea
where the infimum is taken over all $n$-dimensional lattices $L$ in $\R^n$. We call the quantity $\vartheta_s(K)$
the {\em absolute $s$-inhomogeneous minimum}. For $s=1$ we get the
classical absolute inhomogeneous minimum used in
\eqref{bounds_lower}.

Our first theorem shows that
\eqref{bounds_lower} can be generalized to an optimal lower bound on the $s$-Frobenius number.

\begin{theo}\hfill
\begin{itemize}
\item[i)]
Let $n\geq 3$, $s\geq 1$. Then
\begin{equation}
\frob_s({\ve a})\ge \vartheta_s(S_{n-1}) \left(\prod_{i=1}^n a_i\right)^\frac{1}{n-1}-\sum_{i=1}^n a_i
\,.
\label{eq:bounds_generalized_lower}
\end{equation}

\item[ii)] For any $\epsilon>0$ and ${\ve \alpha}=(\alpha_1, \ldots, \alpha_{n-1})\in \R^{n-1}$
with $0<\alpha_1\le \cdots \le \alpha_{n-1}$, there exists a vector ${\ve a}=(a_1, \ldots, a_n)\in \Z^n$
such that
\be \left | \,\alpha_i-\frac{a_i}{a_n}\right |<\epsilon
\,,\;\;\;i=1,2,\ldots,n-1\label{Density}\ee
and
\be \frac{\frob_s({\ve a})+\sum_{i=1}^n a_i}{(a_1 \cdots a_{n})^{1/(n-1)}}<
\vartheta_s(S_{n-1})+\epsilon\,. \label{Sharpness}\ee

\end{itemize}
\label{thm:main}
\end{theo}

Note that the part ii) shows that integer vectors with
relatively small Frobenius numbers approximate any given ``direction'' ${\ve \alpha}\in \R^{n-1}$.

In the general case, the structure of the set ${\mathcal F}_s(A)$, apart from a few special cases, is not well understood. It is well known that the set ${\mathcal F}_1(A)$ can be decomposed into the set of all integer points in the interior of a certain translated cone and a complex complementary set. More recent results  (see \cite{AH_Siam}) attempt to estimate the ``location'' of such a cone in terms of $\det (AA^\intercal)$.
Following the approach of \cite{AH_Siam},  we will define a suitable generalization of the Frobenius number.

Let $C$ be the cone generated by the columns of $A$ and let ${\ve
  v}=A\,{\ve 1}$, where ${\ve 1}$ is the all-$1$-vector.
By the {\em diagonal $s$-Frobenius number} $\dfrob_s(A)$ {\em of} $A$
we understand the minimal $t\ge 0$, such that for all integer vectors ${\ve b}$
in the interior of the translated cone $t{\ve v}+C$ the knapsack polytope $P(A, {\ve b})$ contains at least $s$ integer points, i.e.,
\begin{equation*}
\dfrob_s(A)=\min\left\{t\in\R_{\geq 0} : \mathrm{int}(t{\ve v}+C)\cap\Z^m\subseteq {\mathcal F}_s(A)\right\}.
\end{equation*}
Note that the elements of set ${\mathcal F}_s(A)$ are exactly the vectors ${\ve b}$ such that the integer programming feasibility problem:
\be
\mbox{Does}\; P(A,{\ve b})\; \mbox{contain at least}\; s\;\mbox{integer vectors?}
\label{Knapsack}
\ee
is solved in affirmative. When $s=m=1$, the problem is often called  {\em integer knapsack problem} and is well-known to be NP-complete (see e.g. Section 15.6 in Papadimitriou and Steiglitz \cite{PS}).
Therefore, given any upper bound $\dfrob_s(A)\le f(A)$ that can be
computed in polynomial time, the complimentary set ${\mathcal
  F}_s(A)\setminus ( \mathrm{int}(f(A){\ve v}+C)\cap \Z^m)$ is likely
to have a complex combinatorial structure. Roughly speaking, from the
integer programming point of view, the upper bounds attempt to
determine as small as possible set that contains all ``computationally hard'' instances $A, {\ve b}$.

The next result of this paper gives an upper bound on $\dfrob_s(A)$ in terms of $s$ and $\det (AA^\intercal)$.
\begin{theo} Let $n>m\ge 1$, $s\ge 1$. Then 
\be
\dfrob_s(A)\leq
\frac{n-m}{2}\sqrt{\det(AA^\intercal)} + \frac{(s-1)^\frac{1}{n-m}}{2}\,\left(\sqrt{\det(AA^\intercal)}\right)^\frac{1}{n-m}. \label{upper_bound_for_dFN}
\ee
\label{upper_bound}
\end{theo}
When $s=1$, \eqref{upper_bound_for_dFN} improves Theorem 1.1 in Aliev
and Henk \cite{AH_Siam} by the factor
$\frac{1}{2}n^{-\frac{1}{2}}$. The proof of the theorem is based on estimating $\dfrob_s(A)$
using the $s$-covering radius (\ref{eq:def_covering}).

The problem of estimating the quantity $\dfrob_s(A)$ from below appears to be harder. The natural normalization $(\prod_{i=1}^n a_i)^\frac{1}{n-1}$ in (\ref{eq:bounds_generalized_lower}) reflects the fact that in the case $m=1$
the knapsack polytope $P({\ve a}^\intercal, b)$ is always a simplex. In the general case, $P(A,{\ve b})$ may have a complex combinatorial structure even in the special choice ${\ve b}={\ve v}$. Thus $\dfrob_s(A)$ does not seem to possess a lower bound with {\em simple} natural normalization that applies for all $s$ and matrices $A$ of any size. In this paper we tackle this problem by using lower estimates for the $s$-covering radius. The following theorem gives the first ever lower bound for $\dfrob_s(A)$ in the general case.

\begin{theo} Let $n>m\geq 1$ and $s\geq 1$.
 Then
\bea
\begin{split}
\dfrob_s(A) \geq \frac{1}{n-m+1}\,
\Bigg(s^{\frac1{n-m}}\det(AA^T)^{\frac1{n-m}}\left(\frac{m(A)}{M(A)}\right)^2\sqrt{\frac{n-m}{n}}\\
\quad\quad\quad\quad\quad\quad\quad\quad-(n-m)-2\,\frac{m(A)}{M(A)}\Bigg),\notag \label{lower_bound_for_dFN}
\end{split}
\eea 
where $m(A)$ and $M(A)$ denote the minimal and maximal absolute $m$-minors of $A$.
\label{thm:lower_bound}
\end{theo}

Pleasants et al. \cite{PRS} extensively studied the special case $s=1$ and $m=n-1$ and proved, in particular, that the set ${\mathcal F}_1(A)$ has the following remarkable property.
It is shown in Theorem 6.1 of \cite{PRS} that there exists a unique maximal (w.r.t. inclusion) translated cone ${\ve h}+C$  such that all integral points
in its interior belong to ${\mathcal F}_1(A)$, but there are (infinitely many) integral points on its boundary which are not in ${\mathcal F}_1(A)$.
In general, i.e., for $1<m<n-1$, such a cone is not uniquely determined, see, for instance, \cite{Amosetal} and Section  \ref{m=n-1}.

Theorem 6.1 of \cite{PRS} also describes the location of the unique maximal cone ${\ve h}+C$. We show that this result implies an explicit formula for the diagonal Frobenius number $\dfrob_1(A)$.
Let ${\ve z}=(z_1, z_2, \ldots, z_n)$  be a  primitive integral generator of the kernel of $A$.
\begin{theo}
For $m=n-1$ and  $s=1$ we have
\begin{equation*}
  \dfrob_1(A)= \frac{z^+ z^-}{z^+ +z^-}-1,
\end{equation*}
where $z^+=\max\{z_i: z_i> 0\}$ and $z^-=\max\{|z_i|: z_i< 0\}$.
\label{formula}
\end{theo}

The structure of the paper is as follows. Section \ref{sec:bounds} contains results from the geometry of numbers and new estimates for $\dfrob_s(A)$ in terms of the $s$-covering radius needed in the proofs. In Section \ref{Optimal_lower_bound}, we prove Theorem \ref{thm:main}. The proofs of Theorems  \ref{upper_bound} and \ref{thm:lower_bound} are given in Section \ref{2theorems}.   Finally,
the proof of Theorem \ref{formula} and a new example showing not uniqueness of the maximal cone for $1<m<n-1$  are presented in Section \ref{m=n-1}.

\section{Geometry of numbers and bounds on $\dfrob_s(A)$ in terms of the $s$-covering radius}
\label{sec:bounds}

First we will state results from Minkowski's geometry of
numbers (see \cite{peterbible, GrLek}) needed in this paper.

For an $d$-dimensional lattice $L\subset \R^n$ and an origin-symmetric convex body $K\subset \spn_\R(L)$ the {\em successive minima}  $\lambda_i=\lambda_i(K,L)$
are defined by
\bea \lambda_i(K,L) = \min \left\{ \lambda > 0 : \dim \left(
    \lambda K \cap L \right) \geq i \right\}, 1\leq i\leq d.
\eea

Minkowski's celebrated theorem on successive
minima states (cf.~\cite[Theorem 23.1]{peterbible})
\begin{equation}
 \frac{2^d}{d!}\det(L)\leq \lambda_1\,\lambda_2\cdot\ldots\cdot\lambda_d\,\vol(K)\leq
 2^d\det(L)\,,
\label{eq:second_minkowski}
\end{equation}
where $\vol(K)$ is the $d$-dimensional volume of $K$.

Both inequalities in (\ref{eq:second_minkowski}) can be slightly improved in the special case
of a ball, but since we are mainly not interested in constants
depending on the dimension we do not state these improvements.

There is  a classical relation between the successive minima and
the covering radius due to Jarnik, which we will apply in the
following form
\begin{equation}
\frac{1}{2} \lambda_d(K,L)\leq \mu(K,L)\leq \frac{d}{2}\lambda_d(K,L).
\label{eq:jarnik}
\end{equation}

Next we remark, that the $s$-covering radius
$\mu_s(K,L)$ (see \eqref{eq:def_covering})  may also be described
equivalently as  the smallest positive number $\mu$ such that any translate of $K$
contains at least $s$ lattice points, i.e.,
\begin{equation}\begin{split}
 \mu_s(K,L)=\min\{\mu>0 : \#\{ ({\ve x}+\mu K)\cap L\}\geq s\text{
   for all }\;{\ve x}\in\spn_{\R}(L)\}.
\label{eq:s-covering}
\end{split}
\end{equation}

In \cite{AFH} the  $s$-covering radius was bounded by
\begin{equation}
       s^\frac{1}{d}\left(\frac{\det(L)}{\vol K}\right)^\frac{1}{d}\leq
       \mu_s(K,L) \leq  \mu_1(K,L) + (s-1)^\frac{1}{d}\left(\frac{\det(L)}{\vol(K)}\right)^\frac{1}{d}.
\label{lem:s-covering}
\end{equation}

For $m=1$ there is a nice identity, due to Kannan \cite{Kannan}, between the
Frobenius number $\dfrob_1({\ve a})$ and the covering radius of a simplex. Please see Section \ref{Optimal_lower_bound} for more details.
When $m>1$ no exact formula for $\dfrob_s(A)$ in terms of the $s$-covering radius is known. Here we will prove upper and lower bounds needed for the proofs.

Let $L_A\subset\Z^n$ be the $m$-dimensional lattice generated by the
rows of the given matrix $A\in\Z^{m\times n}$ satisfying the
assumptions \eqref{assumption} and let
\begin{equation*}
 L_A^\perp=\{{\ve z}\in\Z^n : A\,{\ve z}={\ve 0}\}\,.
\end{equation*}
 Then, in particular, we have 
\begin{equation}
        \det L^\perp_A=\det L_A = \sqrt{\det A\,A^\intercal}. 
\label{eq:det}
\end{equation}
Our first lemma gives an upper bound on $\dfrob_s(A)$.

\begin{lemma} Let $1\leq m<n$. Then
\begin{equation*}
          \dfrob_s(A)\leq \mu_s(P(A,{\ve v})-{\ve 1}, L^\perp_A).
\end{equation*}
\label{lem:inhom_min}
\end{lemma}
\begin{proof} Let $t\geq \mu_s(P(A,{\ve v})-{\ve 1}, L^\perp_A)$, and
  let ${\ve b}\in (t\,{\ve v}+C)\cap\Z^m$, i.e., there exists a
  non-negative  vector ${\ve \alpha}\in\R^n_{\geq 0}$ such that ${\ve
    b}=A\,(t\,{\ve 1}+{\ve \alpha})$. By
  \eqref{assumption} i) we know that the columns of $A$ form a
  generating system of the lattice $\Z^m$ (cf.~\cite[Corollary 4.1c]{ASch}). Thus there exists a ${\ve
    z}\in\Z^n$ such that
\begin{equation*}
 {\ve b}=A\,(t\,{\ve 1}+{\ve \alpha})=A\,{\ve z}.
\end{equation*}
 So we have  $P(A,{\ve b})-{\ve z}\subset \spn_\R(L_A^\perp)$ and
it suffices to prove that $P(A,{\ve b})-{\ve z}$ contains at least $s$ integral
points of $L_A^\perp$, for which it is enough to verify
\begin{equation*}
   \mu_s(P(A,{\ve b})-{\ve z}, L_A^\perp)\leq 1.
\end{equation*}
Since the $s$-covering radius is invariant with respect to
translations  and since $P(A,t{\ve v})+{\ve \alpha}\subseteq P(A, {\ve
  b})$ we get
\begin{equation*}
\begin{split}
 \mu_s(P(A,{\ve b})-{\ve z}, L_A^\perp) & = \mu_s(P(A,{\ve b})-(t\,{\ve
   1}+{\ve \alpha}), L_A^\perp) \\
&\leq \mu_s(P(A,t{\ve v})-t\,{\ve
   1}, L_A^\perp)=\mu_s(t\,(P(A,{\ve v})-{\ve
   1}), L_A^\perp)\\
&\leq \frac{1}{t}\mu_s(P(A,{\ve v})-{\ve
   1}, L_A^\perp)\leq 1.
\end{split}
\end{equation*}
\end{proof}

Next, we will obtain a lower bound on $\dfrob_s(A)$ in terms of the $s$-covering radius.
The proof is based on Kannan's approach \cite{Kannan}, but in order to apply it
we have to adjust the setting.
First we extend  the definition of $\dfrob_s(A)$ to all ${\ve w}\in
\mathrm{int}\,C\cap\Z^m$ by defining 
\begin{equation*}
\dfrob_s(A,{\ve w})=\min\{t\geq 0 : \mathrm{int}(t{\ve w}+C)\cap \Z^n \subseteq  {\mathcal F}_s(A)\}.
\end{equation*}
Hence, $\dfrob_s(A)=\dfrob_s(A,{\ve v})$ and it is easy to see that
$\dfrob_s(A,{\ve w})$ is well defined. For instance, since ${\ve w}$
belongs to the interior of the cone, there exists a positive vector
${\ve \rho}\in\R^n_{>0}$ such that ${\ve w}=A{\ve \rho}$.
Denoting by $\rho_{\max}$ and  $\rho_{\min}$ the  maximum and the
minimum of the entries
of ${\ve \rho}$, respectively,  we have
\begin{equation*}
    {\ve w} +C \subseteq  \rho_{\min}\,{\ve v} +C  \text{ and }  {\ve v} +C \subseteq  \frac{1}{\rho_{\max}}\,{\ve w}+C,
\end{equation*}
and thus
\begin{equation}
 \frac{1}{\rho_{\max}} \dfrob_s(A)\leq \dfrob_s(A,{\ve w})\leq
 \frac{1}{\rho_{\min}} \dfrob_s(A).
\label{eq:rel_bound}
\end{equation}

In the following,
let $A=(A_1,A_2)$ with $A_1\in\Z^{m\times m}$ with $\det A_1\ne
0$. Furthermore, let
\begin{equation*}
\mathrm{c}(A_1,A)=\max_{j=1,\ldots,m}\sum_{i=1}^{n-m}\max(\alpha_{i,j},0),
\end{equation*}
where $\alpha_{i,j}$ are the entries of the matrix
$A_1^{-1}\,A_2$. With this notation we have

\begin{lemma} Let $\overline{{\ve v}}= A_1\overline{{\ve 1}}$ be the sum of the
  first $m$ columns of $A$. Let $\widehat{{\ve 1}}=(\overline{{\ve 1}},{\ve
  0})^\intercal\in\Z^n$ be the vector consisting of $m$ ones and $n-m$
zeros.  Then
\begin{equation*}
\lceil \dfrob_s(A, \overline{{\ve v}})\rceil \geq \mu_s(P(A,\overline{{\ve v}})-\widehat{{\ve 1}}, L^\perp_A)-1-c(A_1,A).
\end{equation*}
\label{lem:firstpart}
\end{lemma}

\begin{proof} 
Let ${\ve a}_1,\dots,{\ve a}_n\in\Z^m$ be the columns of $A$, and let
\begin{equation*}
  Q_A=\{A\,{\ve\zeta}: 0< \zeta_i\leq  1\}
\end{equation*}
be the half-open zonotope generated by the columns of $A$. Since  every lattice
point ${\ve b}\in \interior C$, ${\ve b}=\sum_{i=1}^n\rho_i{\ve a}_i$, $\rho_i>0$, can be
written  as
\begin{equation*}
{\ve b}=\sum_{i=1}^n(\rho_i+1-\lceil\rho_i\rceil){\ve a}_i +
\sum_{i=1}^n(\lceil\rho_i\rceil-1){\ve a}_i,
\end{equation*}
 i.e., as a point in $Q_A$
plus a non-negative integral combination of the generators of $C$ we have
\begin{equation*}
\dfrob_s(A, \overline{{\ve v}})  = \min\left\{t\in\R_{\geq 0} :
(t\,\overline{{\ve v}}+Q_A)\cap\Z^m\subset {\mathcal F}_s(A)\right\}.
\end{equation*}
Now for ${\ve l}\in Q_A\cap\Z^m$ let
\begin{equation*}
  t({\ve l})= \min\left\{t\in\Z_{\geq 0} :
t\,\overline{{\ve v}}+{\ve l}\in {\mathcal F}_s(A)\right\}.
\end{equation*}
Then we have
\begin{equation}
\begin{split}
\lceil \dfrob_s(A, \overline{{\ve v}}) \rceil &= \min\left\{t\in\Z_{\geq 0} :
t\,\overline{{\ve v}}+(Q_A\cap\Z^m)\subset {\mathcal F}_s(A)\right\}\\ &=
\max \{t({\ve l}): {\ve l}\in Q_A\cap \Z^m\}.
\end{split}
\label{eq:reduction}
\end{equation}

In analogy
to  \cite{Kannan} we consider $L=L_A^\perp|\R^{n-m}$ and
$P=P(A,\overline{{\ve v}})|\R^{n-m}$, where $\cdot\,|\R^{n-m}$ denotes
the orthogonal projection which forgets the first $m$
coordinates. Hence
\begin{equation*}
L = \{{\ve z}\in\Z^{n-m} : A_1^{-1}A_2\, {\ve z} \in\Z^m\}\text{ and }
P  = \{{\ve x}\in\R^{n-m}_{\geq 0}: A_1^{-1}A_2\, {\ve x} \leq \overline{\ve 1}\}.
\end{equation*}
Since the projections $L_A^\perp$ to $L$ and $P(A,\overline{{\ve v}})$ to $P$ are
bijections  we have
\begin{equation*}
\mu_s(P(A,\overline{{\ve v}})-\widehat{{\ve 1}}, L^\perp_A)=\mu_s(P,L)
\end{equation*}
and we have to prove
\begin{equation}
\mu_s(P,L)\leq  \lceil \dfrob_s(A, \overline{{\ve v}})\rceil +1+c(A_1,A).
\label{eq:to_prove}
\end{equation}


\noindent{\bf Claim}. For ${\ve y}\in\Z^{n-m}$ there exist distinct ${\ve
  z}_1,\dots,{\ve z}_s\in L$ such that ${\ve y }\in {\ve z}_i+(\lceil
\dfrob_s(A, \overline{{\ve v}})\rceil+1)P$, $1\leq i\leq s$.

For the proof of the claim, let ${\ve l}\in Q_A\cap\Z^m$ such that $A_1^{-1} (A_2 {\ve
  y}-{\ve l})\in\Z^m$, which can be found as follows:  write $A_1{\ve \rho}=A_2 {\ve
  y}$ for some ${\ve \rho}\in\R^m$ and set ${\ve
  l}=\sum_{i=1}^m(\rho_i-\lfloor\rho_i\rfloor){\ve a}_i$.

By definition of $t({\ve l})$,
there are $s$ distinct ${\ve x}_i\in\Z^n_{\geq 0}$, $1\leq i\leq s$,
with  $A{\ve x}_i= t({\ve l}){\overline{\ve v}}+{\ve l}$.  We denote
by $\overline{{\ve x}}_i$ and $\widetilde{{\ve x}}_i$ the vectors
consisting of the first $m$ and the last $n-m$ coordinates of ${\ve
  x}_i$, respectively. Then
\begin{equation*}
\begin{split}
A_1^{-1}\,A_2({\ve y}-\widetilde{{\ve x}}_i) &=A_1^{-1}(A_2{\ve
  y}-{\ve l})-A_1^{-1}(A_2\widetilde{\ve x}_i-{\ve l})\\
&=A_1^{-1}(A_2{\ve y}-{\ve l})-t({\ve l})\overline{\ve 1}+\overline{{\ve x}}_i\in\Z^m.
\end{split}
\end{equation*}
Hence ${\ve y}-\widetilde{{\ve x}}_i\in L$, $1\leq i\leq s$, and
observe, that they are also distinct. Since
\begin{equation*}
A_1^{-1}A_2\widetilde{\ve x}_i =A_1^{-1}(t({\ve l})\overline{\ve
  v}+{\ve l}-A_1\overline{\ve x}_i)=
t({\ve l})\overline{\ve 1}+A_1^{-1}{\ve l}-\overline{\ve x}_i \leq (t({\ve
  l})+1)\overline{\ve 1}
\end{equation*}
by the choice of ${\ve l}$, we also have
 $\widetilde{\ve x}_i\in (t({\ve l})+1)P$ and with ${\ve z}_i={\ve
   y}-\widetilde{{\ve x}}_i$, $1\leq i\leq s$,  the claim is proven.

Finally,  we extend this covering property to all points in
$\R^{n-m}$. So let ${\ve y}\in\R^{n-m}$,  and we write ${\ve y}={\ve
  z}+{\ve x}$ with ${\ve z}\in\Z^{n-m}$ and ${\ve
  x}\in[0,1]^{n-m}$. According to the claim there  exist $s$ distinct ${\ve
  z}_1,\dots,{\ve z}_s\in L$ such that ${\ve z}\in {\ve z}_i+(\lceil
\dfrob_s(A, \overline{{\ve v}})\rceil+1)P$, $1\leq i\leq s$. By the
definition of $c(A_1,A)$ we have $A_1^{-1}A_2{\ve x}\leq c(A_1,A){\ve
  1}$ and so ${\ve x}\in c(A_1,A)\,P$. Hence
\begin{equation*}
   {\ve y}\in {\ve z}_i+  (\lceil \dfrob_s(A, \overline{{\ve
       v}})\rceil+1+c(A_1,A))P, \quad 1\leq i\leq s,
\end{equation*}
and \eqref{eq:to_prove} is shown.
\end{proof}

Now, based on this lemma, we can easily derive the desired lower bound for $\dfrob_s(A)$.

\begin{lemma} With the notation of Lemma \ref{lem:firstpart} we have
\begin{equation*}
\dfrob_s(A) \geq \frac1{(n-m+1)}\frac{m(A)}{M(A)}\left(\mu_s(P(A,\overline{\ve v})-{\ve 1},
  L^\perp_A)-2-c(A_1,A)\right).
\end{equation*}
\label{lem:lower_mu}
\end{lemma}
\begin{proof} On account of Lemma \ref{lem:firstpart} we know that
\begin{equation*}
 \dfrob_s(A, \overline{{\ve v}}) \geq \mu_s(P(A,\overline{{\ve v}})-{\ve 1}, L^\perp_A)-2-c(A_1,A),
\end{equation*}
and with the notation in \eqref{eq:rel_bound} we get
\begin{equation*}
 \dfrob_s(A) \geq \rho_{\min}\left(\mu_s(P(A,\overline{{\ve v}})-{\ve 1}, L^\perp_A)-2-c(A_1,A)\right),
\end{equation*}
where $\rho_{\min}$ is the minimal positive entry in a representation
of $\overline{{\ve v}}=A{\ve \rho}$ with ${\ve
  \rho}\in\R^n_{>0}$. Now $\{{\ve x}\in\R^n_{\geq 0} : A{\ve
  x}=\overline{\ve v}\}$ is an $(n-m)$-dimensional polytope in $\R^n$,
and its vertices ${\ve u}$ are determined by $ A{\ve
  u}=\overline{\ve v}$ and $u_i=0$ for some $n-m$
coordinates. Hence, according to Cramer's rule the non-zero
coordinates are the ratio of two $(m\times m)$-minors of $A$, and a
suitable convex combination of $n-m+1$ affinely independent vertices gives us a representation
of $A\overline{\ve v}$ as an interior point. Thus $\rho_{\min}\geq \frac1{(n-m+1)}\frac{m(A)}{M(A)}$.
\end{proof}

\section{Proof of Theorem \ref{thm:main}}
\label{Optimal_lower_bound}

In  \cite{AFH}, a fundamental formula of Kannan \cite{Kannan} relating the
Frobenius number $\frob_1({\ve a})$ with the covering radius was extended to the
$s$-Frobenius number $\frob_s({\ve a})$ and the $s$-covering radius. In order to describe it, let for a given
primitive  positive vector ${\ve a}=(a_1,\dots,a_n)^\intercal
\in\Z^n_{>0}$
\begin{equation*}
S_{\ve a} = \left\{ {\ve x} \in \R_{\geq 0}^{n-1} : a_1\,x_1+\cdots +a_{n-1}\,x_{n-1}\leq 1 \right\}
\end{equation*}
be the $(n-1)$-dimensional simplex with vertices
${\ve 0},\frac{1}{a_i}{\ve e}_i$ where ${\ve e}_i$ is the $i$-th unit vector in
$\R^{n-1}$, $1\leq i\leq n-1$. Furthermore, we consider the following
sublattice of $\Z^{n-1}$
\begin{equation*}
L_{\ve a} = \left\{ {\ve z}\in\Z^{n-1} : a_1\,z_1+\cdots + a_{n-1}\,z_{n-1}\equiv 0 \bmod a_n \right\}.
\end{equation*}
Then
\begin{equation}
       \mu_s(S_{\ve a},L_{\ve a})= \frob_s({\ve a})+a_1+\cdots +a_n,
\label{extension_for_Kannan}
\end{equation}
which for $s=1$ is the above mentioned formula of Kannan.
So the problem to bound $\frob_s({\ve a})$ is equivalent to the study of
$\mu_s(S_{\ve a},L_{\ve a})$.

Set
\bea
\alpha_1=\frac{a_1}{a_n}\,,\ldots\,,\alpha_{n-1}=\frac{a_{n-1}}{a_n}\,.\eea

For
\bea S_{\ve \alpha}=a_n S_{\ve a}=\left\{{\ve x}\in \R^{n-1}_{\ge 0}:
\;\sum_{i=1}^{n-1}\alpha_i x_i \le 1\right\} \,\text{ and }
   L_u=a_n^{-1/(n-1)}L_{\ve a} \eea
we have
%
\be \mu_s (S_{\ve a}, L_{\ve a})=a_n^{1+1/(n-1)}\mu_s (S_{\ve \alpha},
L_u)\,.\label{Vynos} \ee
Observe that $\det L_{\ve a}=a_n$ and hence $\det(L_u)=1$. Consequently,
\be \vartheta_s(S_{\ve \alpha})\le \mu_s (S_{\ve \alpha}, L_u)\,.
\label{S_alpha}\ee
The simplex $(\alpha_1 \cdots \alpha_{n-1})^{1/(n-1)}S_{\ve
\alpha}$ and the standard simplex $S_{n-1}=\{{\ve x}\in \R^{n-1}_{\ge 0}: x_1+\cdots+x_{n-1}\le 1\}$ are equivalent up to a linear transformation
of determinant $1$.
Therefore
\be \vartheta_s(S_{n-1})= \frac{\vartheta_s(S_{\ve \alpha})}{(\alpha_1 \cdots
\alpha_{n-1})^{1/(n-1)}}\,, \label{transform}\ee
and by \eqref{S_alpha}, \eqref{Vynos} and \eqref{extension_for_Kannan}
we have
\bea \vartheta_s(S_{n-1})\le \frac{\mu_s (S_{\ve \alpha}, L_u)}{(\alpha_1
\cdots \alpha_{n-1})^{1/(n-1)}}=\frac{\mu_s (S_{\ve a}, L_{\ve
a})}{a_n^{1+1/(n-1)}(\alpha_1 \cdots \alpha_{n-1})^{1/(n-1)}} \eea
\bea = \frac{\frob_s(a)+a_1+\cdots +a_n}{(a_1\cdots a_n)^{1/(n-1)}},\eea
which shows Theorem \ref{thm:main} i).

One of the main ingredients of the proof of Theorem \ref{thm:main} ii) is Theorem 1.2 in Aliev and Gruber \cite{AG} which is also implicit in Schinzel \cite{NewSL}. For completeness we give below the statement of this result.

\begin{theo}
For any lattice $L$ with basis ${\ve b}_1,\ldots,{\ve b}_{n-1}$,
${\ve b}_i\in\mathbb Q^{n-1}$, $i=1,\ldots,n-1$ and for all
rationals $\alpha_1, \ldots,\alpha_{n-1}$ with
$0<\alpha_1\le\alpha_2\le\cdots\le\alpha_{n-1}\le 1$, there exists
 a sequence
\bea {\ve a}(t)=(a_1(t),\ldots,a_{n-1}(t),a_{n}(t))\in\mathbb
Z^{n}\,, t=1,2,\ldots\,,\eea such that
$\gcd(a_1(t),\ldots,a_{n-1}(t),a_{n}(t))=1$ and the lattice $L_{{\ve
a}(t)}$ has a basis \bea{\ve b}_1(t),\ldots,{\ve b}_{n-1}(t)\eea
with
\be \frac{b_{ij}(t)}{k\,t}=b_{ij}+O\left(\frac{1}{t}\right)\,,\;\;\;
i,j=1,\ldots,{n-1}\,, \label{asympt_aij} \ee
where $k\in\mathbb N$ is such that $k\, b_{ij}, k\,
\alpha_j\,b_{ij}\in\mathbb Z$ for all $i,j=1,\ldots,{n-1}$.
Moreover,
\be a_{n}(t)=\det(L)k^{n-1}t^{n-1}+O(t^{n-2})\label{asympt_h} \ee
and
\be \alpha_i(t):=\frac{a_i(t)}{a_{n}(t)}
=\alpha_i+O\left(\frac{1}{t}\right). \label{asympt_alpha}\ee
\label{main_lemma}
\end{theo}

Following Gruber \cite{Peter}, we say that a
sequence $S_t$ of convex bodies in $\R^{n-1}$ converges to a
convex body $S$ if the sequence of {\em distance functions} of $S_t$
converges uniformly on the unit ball in $\R^{n-1}$ to the distance
function of $S$. For the notion of convergence of a sequence of
lattices to a given lattice we refer the reader to
p.\ 178 of \cite{GrLek}.
\begin{lemma}[see Satz 1 in \cite{Peter}]
Let $S_t$ be a sequence of convex bodies in $\mathbb R^{n-1}$ which
converges to a convex body $S$ and let $L_t$ be a sequence of
lattices in $\mathbb R^{n-1}$ convergent to a lattice $L$. Then
\bea \lim_{t\rightarrow\infty}\mu_s(S_t,L_t)=\mu_s(S,L)\,. \eea
\label{limit}
\end{lemma}

Without loss of generality, we may assume that ${\ve \alpha}\in \Q^{n-1}$ and
\be 0<\alpha_1<\alpha_2<\ldots <\alpha_{n-1}<
1\,.\label{conditions_on_alpha} \ee
For $\epsilon> 0$ we can choose a lattice $L_\epsilon$ of
determinant $1$ with
\be \mu_s(S_{\ve \alpha},L_\epsilon)<\vartheta_s(S_{\ve
\alpha})+\frac{\epsilon(\alpha_1\cdots\alpha_{n-1})^{1/(n-1)}}{2}\,.\label{Step_1}\ee
The inhomogeneous minimum is independent of translation and rational
lattices are dense in the space of all lattices. Thus, by Lemma
\ref{limit}, we may assume that $L_\epsilon\subset\Q^{n-1}$.
Applying Theorem \ref{main_lemma} to the lattice $L_\epsilon$ and
the numbers $\alpha_1, \ldots, \alpha_{n-1}$, we get a sequence
${\ve
a}(t)$, 
satisfying (\ref{asympt_aij}), (\ref{asympt_h}) and
(\ref{asympt_alpha}). Note also that, by
(\ref{conditions_on_alpha}),
\bea 0<a_1(t)<a_2(t)<\ldots <a_{n}(t)\,\eea
for sufficiently large $t$.
%

Observe that the identity (\ref{asympt_alpha}) implies
(\ref{Density}) with $a_i=a_i(t)$, $i=1,\ldots,n$, for $t$ large
enough. Let us show that, for sufficiently large $t$, the inequality
(\ref{Sharpness}) also holds.
Define a simplex $S_{{\ve \alpha}(t)}$ and a lattice $L_t$ by
\bea S_{t}=a_n(t)S_{{\ve
a}(t)}=\{(x_1,\ldots,x_{n-1})\in \R^{n-1}_{\ge 0}:\;
\sum_{i=1}^{n-1}\alpha_i(t) x_i \le 1\} \,, \eea
\bea L_t=a_n(t)^{-1/(n-1)}L_{{\ve a}(t)}\,. \eea

By (\ref{asympt_aij}) and (\ref{asympt_h}),
the sequence 
$L_t$ converges to the lattice $L_\epsilon$.
%
Next, the point ${\ve p}=(1/(2n), \ldots, 1/(2n))$ is an inner point
of the simplex $S_{\ve \alpha}$ and all the simplicies $S_{t}$ for sufficiently large $t$. By (\ref{asympt_alpha}) and
Lemma \ref{limit}, the sequence $\mu_s(S_{t}-{\ve p},
L_t)$ converges to $\mu_s(S_{\ve \alpha}-{\ve p}, L_\epsilon)$.
Here we consider the sequence $\mu_s(S_{t}-{\ve p},
L_t)$ instead of $\mu_s(S_{t}, L_t)$ because the
distance functions of the family of convex bodies in Lemma \ref{limit}
need to converge on the unit ball.
Now, since the inhomogeneous minimum is independent of translation,
the sequence $\mu_s(S_{t}, L_t)$ converges to
$\mu_s(S_{\ve \alpha}, L_\epsilon)$. Consequently, by
(\ref{asympt_alpha}),
\bea \frac{\mu_s(S_{t},
L_t)}{(\alpha_1(t)\cdots\alpha_{n-1}(t))^{1/(n-1)}} \rightarrow
\frac{\mu_s(S_{\ve \alpha},
L_\epsilon)}{(\alpha_1\cdots\alpha_{n-1})^{1/(n-1)}}\,,\;\;\mbox{as}\;t\rightarrow\infty\,,\eea
and, by (\ref{extension_for_Kannan}), 
(\ref{Step_1}) and (\ref{transform}),
\bea \frac{\frob_s({\ve a}(t))+\sum_{i=1}^n a_i(t) }{(a_1(t)\cdots
a_{n}(t))^{1/(n-1)}}=\frac{\mu_s (S_{t},
L_t)}{(\alpha_1(t)\cdots
\alpha_{n-1}(t))^{1/(n-1)}}<\vartheta_s(S_{n-1})+\epsilon\,\eea
for sufficiently large $t$. This concludes the proof of Theorem
\ref{thm:main}

\section{Proof of Theorem \ref{upper_bound} and Theorem \ref{thm:lower_bound}}
\label{2theorems}

In Section \ref{sec:bounds} we have already proven the bounds for $\dfrob_s(A)$
in terms of the $s$-covering radius and it remains now to bound the $s$-covering radius itself.

\begin{lemma}
The inequality
\begin{equation*}
\mu_1(P(A,{\ve v})-{\ve 1}, L_A^\perp)\le\, \frac{n-m}{2}\sqrt{\det(AA^\intercal)}
\end{equation*}
holds.
\label{improved_upper_bound}
\end{lemma}

\begin{proof}
Let $C^{n}=[-1,1]^n$ and $K=C^{n} \cap \lin L^\bot_A$.
By a well-known result of Vaaler \cite{Vaaler}, any $k$-dimensional section of the
cube $C^{n}$ has $k$-volume at least $2^k$. Since ${\ve 1}\in P(A, {\ve v})$, the
latter observation implies that the polytope $P(A, {\ve v})-{\ve 1}$ contains an
$(n-m)$-dimensional section $Q$ of the cube $C^n$ and, in particular,
\begin{equation}
\vol_{n-m}(P(A, {\ve v}))\geq \vol_{n-m}(Q)\ge 2^{n-m},
\label{eq:volume_poly}
\end{equation}
and by  \eqref{eq:jarnik}
\begin{equation}
  \mu_1(P(A,{\ve v})-{\ve 1}, L_A^\perp)\leq \mu_1(Q,
 L_A^\perp)\leq \frac{n-m}{2}\lambda_{n-m}(Q,L_A^\perp).
\label{eq:h1_new}
\end{equation}
All
vectors of the lattice $L_A^\perp$ are integral vectors, thus
 $\lambda_i(Q,L_A^\perp)\geq 1$, $1\leq i\leq n-m$. Hence
from \eqref{eq:second_minkowski} and \eqref{eq:volume_poly} we get
\begin{equation}
\lambda_{n-m}(Q,L_A^\perp)\leq \det
L_A^\perp,
\end{equation}
 and with \eqref{eq:h1_new} we are done.
\end{proof}

One can also  obtain a refinement of the bound
above.

\begin{lemma} The inequality
\begin{equation*}
\mu_1(P(A,{\ve v})-{\ve 1}, L_A^\perp)\le\, \frac{n-m}{2}\,M(A)
\end{equation*}
holds.
\label{lem:improvi}
\end{lemma}
\begin{proof}
For an $m$-subset $J\subset\{1,\dots,n\}$ let $A_J$ be the submatrix
of $A$ consisting of the columns with index set $J$.
 Let  $I=\{1,\dots,m\}$ and  without loss of generality let $\det
 A_I\ne 0$.  Then, for $k=m+1,\dots,n$,  the
  $1$-dimensional subspace $\{{\ve x}\in\R^{n}: A\,{\ve x}=0,\,x_i=0\text{ for }
  i\notin I\cup\{k\}\}$ is
  generated by the vector ${\ve z}_k\in\Z^{n}$ with
\begin{equation*}
  {\ve z}_{ki}=(-1)^{m-i}\det A_{\{1,\dots,i-1,i+1,\dots,m,k\}},\, 1\leq i\leq
  m,\quad {\ve z}_{kk}=-\det A_I,
\end{equation*}
and zero otherwise. Clearly,  ${\ve z}_k\in L_A^\bot$ for $k=m+1,\dots,n$.

We have to show that
\begin{equation*}
\left({\ve x} + \frac{n-m}{2}M(A) \left(P(A,{\ve v})-{\ve 1}\right)\right)\cap L_A^\perp
\neq \emptyset
\end{equation*}
for all ${\ve x}\in\lin L_A^\perp$.

To this end let ${\ve c}=\frac{n-m}{2}M(A){\ve 1}$. Due
to the maximality of $M(A)$ we have  ${\ve c}\pm \frac{(n-m)}{2}\,{\ve
  z}_k\in\R^n_{\geq 0}$ and so  ${\ve c}\pm \frac{(n-m)}{2}\,{\ve
  z}_k\in \frac{n-m}{2}M(A)P(A,{\ve v})$ for $k=m+1,\dots, n$. Hence
the convex hull $C$ of these points, i.e.,
\begin{equation*}
       C= \left\{ {\ve c}+\frac{(n-m)}{2}\sum_{k=m+1}^{n} \mu_{k}\,\pm{\ve z}_{k}:
          \sum_{k=m+1}^n\mu_k=1,\mu_k\geq 0  \right\}
\end{equation*}
is contained in $\frac{n-m}{2}M(A)P(A,{\ve v})$.
Now this cross-polytope $C$  contains the parallelepiped
\begin{equation*}
\left\{ {\ve c}\pm\rho_{m+1} \frac{1}{2}\,{\ve z}_{m+1}\pm
          \cdots \pm\rho_{n} \frac{1}{2}\,{\ve z}_{n} :
          \rho_k\in[0,1] \right\},
\end{equation*}
which shows that $\frac{n-m}{2}M(A) \left(P(A,{\ve v})-{\ve 1}\right)$
contains  a translation of the lattice (w.r.t. $L_A^\perp$) parallelepiped
$U=\{\sum_{k=m+1}^m \rho_k {\ve z}_{k}: 0\leq \rho_k\leq 1\}$.
Since  any
translated lattice parallelepiped in $\lin L_A^\perp$ must contain a
lattice point of $L_A^\perp$  we know that ${\ve x} +
\frac{n-m}{2}M(A) \left(P(A,{\ve v})-{\ve 1}\right)$ contains a point
of $L_A^\perp$ for all ${\ve x}\in \lin L_A^\perp$.
\end{proof}

Now we are ready to complete the proof of Theorem \ref{upper_bound}

\begin{proof}[Proof of Theorem \ref{upper_bound}] By Lemma
  \ref{lem:inhom_min} and \eqref{lem:s-covering} we have
\begin{equation*}
\begin{split}
  \dfrob_s(A)&\leq \mu_s(P(A,{\ve v})-{\ve 1}, L_A^\perp)\\&\leq
\mu_1(P(A,{\ve v})-{\ve 1},L_A^\perp) + (s-1)^\frac{1}{n-m}\left(\frac{\det L_A^\perp}{\vol(P(A,{\ve v}))}\right)^\frac{1}{n-m}.
\label{aux1}
\end{split}
\end{equation*}
Together with Lemma \ref{improved_upper_bound} and the lower bound on $\vol(P(A,
{\ve v}))$ given in \eqref{eq:volume_poly} we obtain the inequality \eqref{upper_bound_for_dFN}.
\end{proof}

Observe that using Lemma \ref{lem:improvi} instead of Lemma
\ref{improved_upper_bound}  results in the  better but less
nice bound
\begin{equation*}
 \dfrob_s(A) \leq
\frac{n-m}{2}M(A) + \frac{(s-1)^\frac{1}{n-m}}{2}\,\left(\sqrt{\det(AA^\intercal)}\right)^\frac{1}{n-m}.
\end{equation*}

Finally we give the proof of  Theorem \ref{thm:lower_bound}

\begin{proof}[Proof of Theorem \ref{thm:lower_bound}] In view of
  Lemma \ref{lem:lower_mu} and \eqref{lem:s-covering}  we have
\begin{equation*}
\dfrob_s(A) \geq \frac{1}{n-m+1}\,
\frac{m(A)}{M(A)}\left(s^\frac{1}{n-m}\left(\frac{\sqrt{\det AA^\intercal}}{\vol
    (P(A,\overline{\ve v}))}\right)^\frac{1}{n-m}-2-c(A_1,A)\right).
\end{equation*}
Thus we need an upper  bound on $c(A_1,A)$ and on
$\vol(P(A,\overline{{\ve v}}))$. With Cramer's rule each entry of
$A_1^{-1} A_2$ is of the form $\det A_I/\det A_1$, where $A_I$ is a
certain $m\times m$ minor. Consequently, we have 
\begin{equation}
c(A_1,A) \leq (n-m)\frac{M(A)}{m(A)}.
\label{eq:bound_c}
\end{equation}
Let ${\ve u}\in\R^n$ be a vertex of $P(A,\overline{\ve v})$. As in the proof of
Lemma \ref{lem:lower_mu} we may argue that each of the $m$ non-zero
entries of ${\ve u}$ is bounded by $M(A)/m(A)$. So $P(A,\overline{\ve
  v})$ is contained in an $n$-dimensional cube of edge length $M(A)/m(A)$.
By \cite{Ball}, the largest volume of an $(n-m)$-dimensional section of a
cube of dimension $n$ with edge length $\sigma$
is $\left(\sigma\sqrt{\frac{n}{n-m}}\right)^{{n-m}}$ which finishes the proof.
\end{proof}

\section{The case $m=n-1$. Proof of Theorem \ref{formula}}
\label{m=n-1}

Recall that an integer vector ${\ve z}=(z_1,z_2,\ldots,z_n)$ is called {\em primitive} if $\gcd(z_1,z_2,\ldots,z_n)=1$.
In the case $m=n-1$ the lattice $L_A^\perp$ has dimension one, so
it is generated by a primitive vector ${\ve z}\in\Z^n$.
Following \cite{PRS}, we put
\begin{equation*}
   {\bf u}=\sum_{i: z_i\geq 0} z_i\,{\bf a}_i=-\sum_{i: z_i\leq 0} z_i\,{\bf a}_i.
\end{equation*}
It is shown in \cite[Theorem 6.1]{PRS} that all points in
the interior of $C+{\ve u}-{\ve v}$ admit a non-negative integral
representation, but that, all the lattice points in the facets, which are
contained in the lattice generated  by those ${\ve a}_i$'s contained in the
facet, are 
non-reachable. Hence we have
\begin{equation*}
\dfrob_1(A)=\min\{\gamma\in\R_{\geq 0}: \gamma\,{\ve v}\in C+{\ve u}-{\ve v}\}.
\end{equation*}
Therefore
\begin{equation*}
\dfrob_1(A)=\min\{\gamma\in\R_{\geq 0}: \exists\, \rho\in\R\quad
 \text{s.t.}\quad  (\gamma+1)\,{\ve 1} -(z_i)^\intercal_{z_i\geq 0} + \rho\,{\ve z}\geq 0\}.
\end{equation*}
Here $(z_i)^\intercal_{z_i\geq 0}\in\R^n$ denotes the vector
consisting of the nonnegative entries of ${\ve z}$ (the others are
$0$). So we may write ${\ve z}=(z_i)^\intercal_{z_i\geq
  0}+(z_i)^\intercal_{z_i\leq 0}$ and we are interested in the minimal
$\gamma\in\R_{\geq 0}$ such that there exists a $\rho$ with
\begin{equation*}
\gamma+1\geq (1-\rho)\, z_i, \text{ for } z_i\geq 0 \quad \text{ and }
 \gamma+1\geq \, -\rho\,z_i, \text{ for }z_i\leq 0.
\end{equation*}

Recall that we put $z^+=\max\{z_i: z_i> 0\}$ and
$z^-=\max\{|z_i|: z_i< 0\}$.
 Then
\begin{equation*}
\gamma+1= \min_\rho \max\{(1-\rho)\, z^+, \rho\,z^-\}.
\end{equation*}
The first function is decreasing in $\rho$ whereas the second is
increasing in $\rho$. Both coincide for $\rho=z^+/(z^++z^-)$ and so we
have proved Theorem \ref{formula}.

\smallskip
As mentioned in the introduction, it was also shown in \cite{PRS} that
the $C+{\ve u}-{\ve v}$ is the unique maximal cone with the property that all points in
the interior  admit a non-negative integral
representation, but infinitely many on the boundary do not have such a
representation. This seems to be a very particular property of the
case $m=n-1$ as it was already pointed out in \cite{Amosetal}. We want
 to conclude the paper  with another example showing that, in general,
 we have more than one maximal cone.

\begin{figure}[htbp]
\includegraphics[scale=0.7]{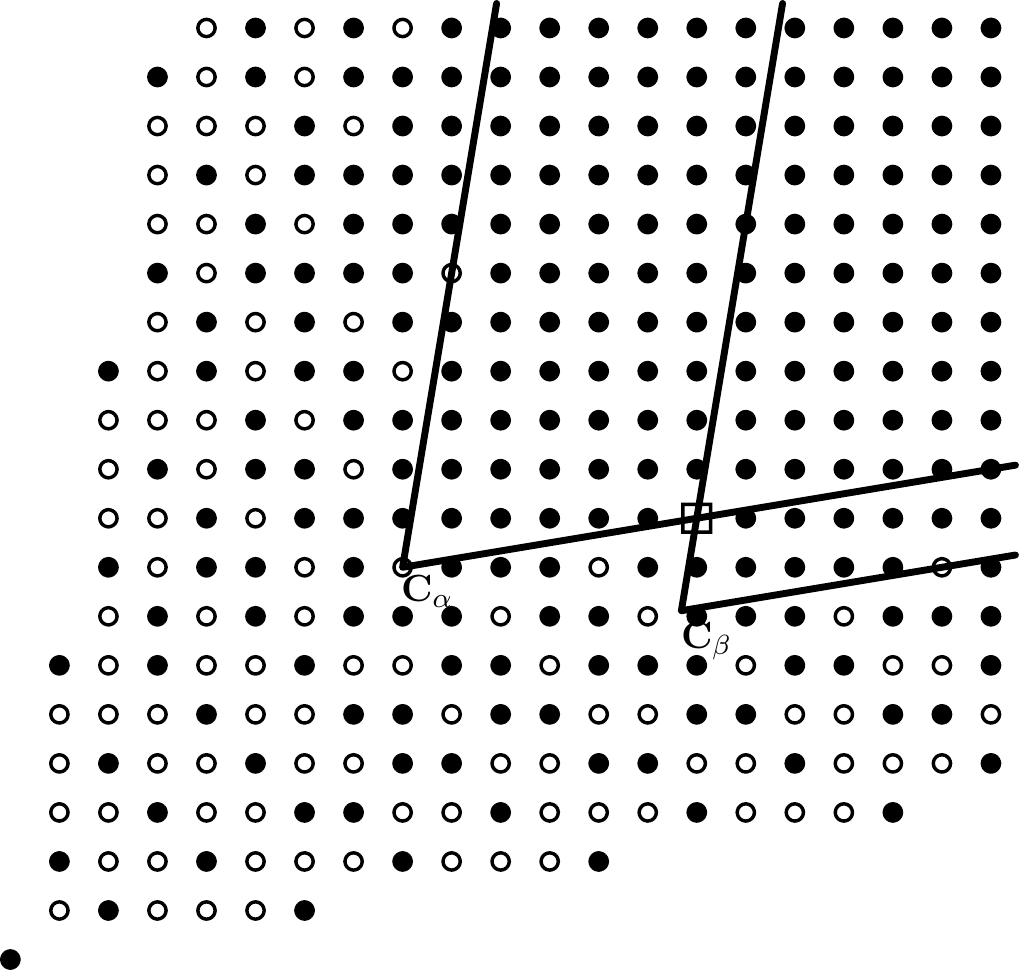}
\caption{Non-uniqueness
  for $m=n-2$.}
\label{fig1}
\end{figure}

For $k\in\N$ let
\begin{equation*}
A=\begin{pmatrix}1&2&1&3k\\2&1&3k&1\end{pmatrix}.
\end{equation*}
Then the cone $C$ (see Figure \ref{fig1} for $k=2$)  is generated by the two vectors $(1,3k)^\intercal$ and
$(3k,1)^\intercal$.
By elementary calculations, we get, that $\binom{9k-4}{6k-3}\not\in {\mathcal F}_1(A)$
but all other integral points in $\binom{9k-4}{6k-3}+\conv\{\binom{0}{0},\binom{1}{3k},\binom{3k}{1},\binom{3k+1}{3k+1}\}$
are contained in ${\mathcal F}_1(A)$.

This implies, that all integral points but $\binom{9k-4}{6k-3}$ in
\[\binom{9k-4}{6k-3}+C=\left\{\binom{x}{y}\in\R^2:
\begin{array}{rl}3kx-y&\geq 3(3k-1)^2,\\-x+3ky&\geq 2(3k-1)(3k-2)\end{array}\right\}\] are
contained in ${\mathcal F}_1(A)$. Now let
 \[\alpha=\min\{t\in\Z:
\left\{\binom{x}{y}\in\Z^2: \begin{array}{rl}3kx-y&> t\\ -x+3ky&>
    2(3k-1)(3k-2)\end{array}\right\}\subset {\mathcal F}_1(A)\}\]
and
\[\beta=\min\{t\in\Z:
\left\{\binom{x}{y}\in\Z^2: \begin{array}{rl}3kx-y&> 3(3k-1)^2\\ -x+3ky&> t\end{array}\right\}\subset {\mathcal F}_1(A)\}.\]
Then it is clear, that $\alpha< 3(3k-1)^2$ and $\beta<2(3k-1)(3k-2)$. This gives us two cones, namely
\[
\begin{split}
C_\alpha &=\left\{\binom{x}{y}\in\R^2:\begin{array}{rl}3kx-y&\geq \alpha,\\-x+3ky&\geq 2(3k-1)(3k-2)\end{array}\right\}\\
C_\beta &=\left\{\binom{x}{y}\in\R^2:\begin{array}{rl}3kx-y&\geq
    3(3k-1)^2,\\-x+3ky&\geq \beta\end{array}\right\},
\end{split}
\]
that are maximal w.r.t.\ inclusion with the property, that all
integral points in their interiors are in ${\mathcal F}_1(A)$.
By symmetries, the same is true for $\binom{6k-3}{9k-4}$.

For $k=2$ the two maximal cones $C_\alpha$ and $C_\beta$ are depicted  in Figure \ref{fig1}. Black dots
are in ${\mathcal F}_1(A)$, circles/squares with white interior are
not. The square represents the point
$\binom{14}{9}=\binom{9k-4}{6k-3}$ for $k=2$.

\end{document}